\documentclass[11pt]{amsart}
\usepackage[margin=1in]{geometry}

\usepackage{MnSymbol}
\usepackage{mathrsfs}
\usepackage{colonequals}
\usepackage{calligra}

\usepackage{tikz-cd}

\makeatletter
\newcommand{\oset}[3][0ex]{%
  \mathrel{\mathop{#3}\limits^{
    \vbox to#1{\kern-2\ex@
    \hbox{\(\scriptstyle#2\)}\vss}}}}
\makeatother

\newcommand{\pp}{\mathbb{P}}
\newcommand{\zz}{\mathbb{Z}}

\renewcommand{\O}{\mathcal{O}}

\newcommand{\sH}{\mathscr{H}}

\newcommand{\et}{\text{\'et}}
\newcommand{\pr}{\text{pr}}

\newcommand{\Ext}{\operatorname{Ext}}
\newcommand{\Pic}{\operatorname{Pic}}

\newcommand{\rk}{\operatorname{rk}}

\newcommand{\defi}[1]{\textsf{#1}} 

\newtheorem{thm}{Theorem}[section]
\newtheorem{lem}[thm]{Lemma}
\newtheorem{prop}[thm]{Proposition}

\newtheorem{conj}[thm]{Conjecture}

\theoremstyle{definition}

\theoremstyle{remark}
\newtheorem{rem}{Remark}

\newcommand{\leqor}{\underset{{\scriptscriptstyle (}-{\scriptscriptstyle )}}{<}}

\title{Generic Beauville's Conjecture}
\author{Izzet Coskun}
\address{Department of Mathematics, Statistics, and CS \\
University of Illinois at Chicago, Chicago IL 60607}
\email{icoskun@uic.edu}
\author{Eric Larson}
\address{Department of Mathematics, Brown University}
\email{elarson3@gmail.com}
\author{Isabel Vogt}
\address{Department of Mathematics, Brown University}
\email{ivogt.math@gmail.com}

\thanks{During the preparation of this article, I.C.\ was supported
by NSF FRG grant DMS-1664296 and NSF grant DMS-2200684,  E.L. was supported by
NSF  grants DMS-1802908 and DMS-2200641, and I.V. was supported by NSF grants DMS-1902743 and DMS-2200655.}

\begin{document}

\begin{abstract}
    Let $\alpha\colon X \to Y$ be a finite cover of smooth curves.  Beauville conjectured that the pushforward of a general vector bundle under \(\alpha\) is semistable if the genus of \(Y\) is at least \(1\) and stable if the genus of \(Y\) is at least \(2\).  We prove this conjecture if the map \(\alpha\) is general in any component of the Hurwitz space of covers of an arbitrary smooth curve \(Y\).
\end{abstract}

\maketitle

\section{Introduction}

  Motivated by the study of the theta linear series on the moduli spaces of vector bundles on curves, Beauville in \cite{b00} (see also \cite[Conjecture 6.5]{b06}) made the following celebrated conjecture: 

\begin{conj}[Beauville]\label{conj-main}
Let $\alpha\colon X \to Y$ be a finite morphism between smooth irreducible projective curves, and let $V$ be a general vector bundle on $X$. Then $\alpha_* V$ is stable if the genus of \(Y\) is at least \(2\) and semistable if the genus of \(Y\) is \(1\).
\end{conj}
\noindent We prove Beauville's Conjecture when $Y$ is an {\em arbitrary} smooth irreducible projective curve and $X$ is a general member of {\em any} component of the Hurwitz space of genus $g$ degree $r$ covers of $Y$.

\subsection*{Statement of results} Let $\alpha\colon X \to Y$ be a finite map of degree $r$ from a smooth irreducible projective curve $X$ of genus $g$ to a smooth irreducible projective curve $Y$ of genus $h$. We always work  over an algebraically closed field of characteristic 0 or greater than $r$. 

For a vector bundle $V$ on a curve $X$,  the \defi{slope} $\mu(V)$ is defined by $\mu(V) = \frac{\deg(V)}{\rk(V)}$. The bundle $V$ is called \defi{(semi)stable} if,  for every proper subbundle $W$, we have $\mu(W) \leqor \mu(V)$. Semistable bundles satisfy nice cohomological and metric properties and form projective moduli spaces. Consequently, determining the stability of naturally defined bundles is an important and fundamental problem.

Let $\sH_{r,g}(Y)$ denote the Hurwitz space parameterizing smooth connected degree $r$ genus $g$ covers of $Y$. In general, $\sH_{r,g}(Y)$ is reducible, and when \(g > r(h-1) + 1\), the irreducible components correspond to subgroups of the (\'etale) fundamental group $\pi_1(Y)$ of index diving $r$. With this notation, 
our main theorem is the following.

\begin{thm}\label{thm:gen}
Let \(Y\) be any smooth irreducible projective curve of genus \(h\).  Let \(\alpha \colon X \to Y\) be a general morphism in any component of \(\sH_{r,g}(Y)\).  Let \(V\) be a general vector bundle of any degree and rank on \(X\).
\begin{enumerate}
    \item If \(h=1\), then \(\alpha_*V\) is semistable.
    \item If \(h\geq 2\), then \(\alpha_*V\) is stable.
\end{enumerate}
\end{thm}

\begin{rem}
It may happen that for special $V$, the bundle $\alpha_* V$ is not semistable. For example, $\alpha_* \O_X$ has $\O_Y$ as a direct summand. When the map $\alpha$ is ramified, $\O_Y$ destabilizes $\alpha_* \O_X$ (see \cite{clv-tschirnhausen}).
\end{rem}

\subsection*{History of the problem} Beauville made Conjecture \ref{conj-main} in an unpublished note dating to 2000 \cite{b00}. In the same note, Beauville proved the conjecture if 
\begin{enumerate}
    \item $\alpha$ is \'etale \cite[Propostion 4.1]{b00}; or
    \item $r < g (\sqrt{3} + 1) -1$ \cite[Corollary 3.4]{b00}; or
    \item $V$ is a line bundle with $|\chi(V)| \leq g + \frac{g^2}{r}$ \cite[Proposition 3.2]{b00}.
\end{enumerate}
It is an elementary observation that when \(h=0\), the pushforward of a general vector bundle splits as \(\bigoplus_{i=1}^r \O_{\pp^1}(a_i)\), where \(|a_i-a_j| \leq 1\) for every \(i,j\) (see \cite[\S 1]{b00}).

Beauville, Narasimhan, and Ramanan \cite{bnr} earlier proved that a general vector bundle $V$ of degree $d$ and rank $r$ on $Y$ arises as the pushforward of a line bundle from some cover of degree $r$. Hence, there exists covers of $Y$ for which Beauville's Conjecture is true.

Mehta and Pauly \cite{mp07} proved that if $\alpha$ is the Frobenius morphism, $h\geq 2$, and $V$ is semistable, then $\alpha_* V$ is semistable.

\subsection*{Strategy}
Recall that $\alpha\colon X \to Y$  is called  \defi{primitive} if the map   $\alpha_*\colon \pi_1(X) \to \pi_1(Y)$ induced on (\'etale) fundamental groups  is surjective. Every degree $r$ cover $\alpha \colon X \to Y$ factors into  a primitive map $\alpha^{\pr} \colon X \to Y'$  followed by an \'etale map $\alpha^{\et}\colon Y' \to Y$, where $Y'$ is the \'etale cover associated with the subgroup $\alpha_*\pi_1(X) \subset \pi_1(Y)$.

We  prove Theorem \ref{thm:gen} by specializing \(X\) to a nodal curve. Let $\alpha\colon X \to Y$ be a general element in any component of $\sH_{r,g}(Y)$. Let $\alpha= \alpha^{\pr} \circ \alpha^{\et}$ be the primitive-\'etale factorization of $\alpha$. Let $\alpha^{\et}$ and $\alpha^{\pr}$ have degrees $r'$ and  $s= \frac{r}{r'}$, respectively. Let $\beta_0\colon X_0 \to Y'$  be a degree $s$ cyclic \'etale cover of $Y'$. The resulting map $\alpha_0 \colon X_0 \to Y$ is \'etale and Conjecture \ref{conj-main} holds for the map $X_0 \to Y$ by \cite[Proposition 4.1]{b00}. 

Let $p_{j}$ and $p_{j}'$ be points on $X_0$ contained in a fiber of $\beta$  such that the cyclic action takes $p_{j}$ to $p_{j}'$.  We identify the appropriate number of pairs  $p_{j}$ and $p_{j}'$ on $X_0$  to form a nodal curve $X_1$ of genus $g$. Let $\nu\colon X_0 \to X_1$ be the normalization map. The induced map $\beta_1\colon X_1 \to Y'$ is primitive (see Proposition \ref{prop-primitive}), and so $\alpha_1= \alpha^{\et} \circ \beta_1\colon X_1 \to Y$ is in the same irreducible component of $\sH_{r,g}(Y)$ as $X$ (see Lemma \ref{lem-hurwitz}). For a general bundle $V$ on $X_0$, the pushforward ${\alpha_0}_*V={\beta_1}_* (\nu_* V)$ is stable if $h \geq 2$ and semistable if $h=1$. Finally,  we observe that  $\nu_* V$ is a limit of vector bundles on nearby deformations of $X_1$ (see Proposition \ref{pushforward-is-limit}). Together with the openness of (semi)stability, this proves Theorem \ref{thm:gen}.

\subsection*{Acknowledgements} We would like to thank Arnaud Beauville, Anand Deopurkar, Aaron Landesman, Daniel Litt, Chien-Hao Liu,  Anand Patel and Prof. Shing-Tung Yau  for invaluable discussions. 
We are grateful to Mathematisches Forschungsinstitut Oberwolfach for providing an excellent work environment during the ``Recent trends in algebraic geometry'' workshop in June 2023.
%Finally, we thank the anonymous referees for helpful feedback that have improved the paper. 

\section{Preliminaries}

\subsection{Basic facts} Let $\alpha\colon X \to Y$ be a finite map of degree $r$ from a smooth irreducible projective curve $X$ of genus $g$ to a smooth irreducible projective curve $Y$ of genus $h$. Since the characteristic is 0 or greater than $r$, the map $\alpha$ is separable. By the Riemann--Hurwitz formula
$$2g-2 = r(2h-2) + b,$$ where $b$ is the degree of the ramification divisor. In particular,
$g \geq r(h-1) + 1$ with equality if and only if $\alpha$ is \'etale. 

If $V$ is a vector bundle of rank $s$ and degree $d$ on $X$, then $\alpha_*(V)$ is a vector bundle of rank $rs$ on $Y$. Using the fact that $\chi(V)= \chi(\alpha_*V)$ and the Riemann--Roch formula, we compute the degree $d'$ of $\alpha_*(V)$ via
$$d+s(1-g) = \chi(V) = \chi(\alpha_*V)= d' + rs(1-h).$$
We conclude that $d' = d+ s(1-g) - rs(1-h).$

\subsection{The primitive-\'etale factorization}

Let $\sH_{r,g}(Y)$ denote the Hurwitz space parameterizing smooth connected degree $r$ genus $g$ covers of $Y$. If $g < r(h-1)+1$, then $\sH_{r,g}(Y)$ is empty. If $g=r(h-1)+1$, then the degree $r$ covers of genus $g$ are \'etale and there are finitely many. In general, the Hurwitz space $\sH_{r,g}(Y)$ is not irreducible. The following lemma characterizes the irreducible components.

\begin{lem}\label{lem-hurwitz}
    Let \(Y\) be a smooth and irreducible curve of genus $h$ defined over an algebraically closed field of characteristic \(0\) or larger than \(r\).  Let $g > r(h-1) + 1$. Then the components of \(\sH_{r,g}(Y)\) are in bijection with subgroups of \(\pi_1(Y)\) of index dividing \(r\). 
\end{lem}

\begin{proof}
Let $\sH$ be an irreducible component of the Hurwitz scheme $\sH_{r,g}(Y)$. Given $\alpha \colon X \to Y$ in $\sH$, the subgroup $\alpha_* \pi_1(X)$ of $\pi_1(Y)$ has index dividing $r$. Since this is a discrete invariant and is constant in irreducible families, $\alpha_* \pi_1(X)$ is an invariant of $\sH$. 

Conversely, given a subgroup $G \subset \pi_1(Y)$ of index $r^\et$ dividing $r$, up to isomorphism there is a unique \'etale cover $\delta\colon Y' \to Y$ corresponding to $G$ of degree $r^\et$ and genus \(h' = r^\et(h-1) + 1\). Let \(r^\pr = r/r^\et\).  Given the inequality 
\[g > r(h-1) + 1 = r^\pr(h'-1) + 1,\] 
there exists a genus $g$ primitive cover $\gamma \colon X \to Y'$ of degree $r^\pr$.  For any such cover, we obtain an element of $\sH_{r,g}(Y)$ by taking $\alpha= \delta \circ \gamma$. Furthermore, $\alpha_* \pi_1(X) = G$. On the other hand, if $\gamma \colon X \to Y'$ is not primitive but $\gamma_* \pi_1(X)$ has index $s$ in $\pi_1(Y')$, then $\alpha_* \pi_1(X)$ has index $s r^\et$ in $\pi_1(Y)$ and cannot be $G$. We conclude that if $\alpha_*\pi_1(X)=G$, then $\alpha$ must factor as the composition of $\delta$ and a primitive cover of $Y'$.  By results of Clebsch \cite{clebsch}, Fulton \cite{fulton}, and Gabai--Kazez \cite{gk90}
(see \cite[Proposition 2.2]{clv-tschirnhausen}), the Hurwitz scheme parameterizing genus $g$ degree $r^\pr$ primitive covers of $Y'$ is irreducible. We conclude that there is a bijection between irreducible components of $\sH_{r,g}$ and subgroups of $\pi_1(Y)$ of index dividing $r$.
\end{proof}

\subsection{The \'etale case}\label{sec-etale} We briefly recall Beauville's proof of Conjecture \ref{conj-main}  \cite[Proposition 4.1]{b00} in the \'etale case (see also the proof of \cite[Proposition 1.3]{clv-tschirnhausen}).

First, we show that it suffices to consider the case of line bundles. Given a vector bundle $V$ on $X$ of degree $d$ and rank $s$, let $\delta \colon Z \to X$ be an \'etale cover of degree $s$. If $L$ is a line bundle of degree $d$ on $Z$, then $\delta_* L$ is a vector bundle of rank $s$ and degree $d$ on $X$. Hence, if we prove Conjecture \ref{conj-main} for (\'etale) maps in the case of line bundles, it follows for (\'etale) maps in higher rank as well.

Let $\alpha \colon X \to Y$ be \'etale and let $\rho \colon Z \to Y$ be the Galois cover associated to $\alpha$ with Galois group $G$. Let $\Sigma$ be the set of $Y$-morphisms $\sigma \colon Z \to X$. Then  
\[W \colonequals \rho^* \alpha_* L \cong  \bigoplus_{\sigma \in \Sigma} \sigma^* L.\]  The pullback by $\rho$ of any destabilizing subbundle of $\alpha_* L$ would destabilize $W$. Hence, $\alpha_* L$ is semistable for {\em every} line bundle $L$ on $X$.

If $\alpha_* L$ has a proper subbundle $F$ of the same slope as $\alpha_* L$, then $\rho^* F$ is a $G$-invariant subbundle of $W$.  Since the category of semistable bundles of a fixed slope is abelian with simple objects stable bundles, $\rho^* F \cong \oplus_{\sigma \in \Sigma'} \sigma^* L$ for some $\Sigma' \subset \Sigma$. 
Since $G$ acts transitively on $\Sigma$, it suffices to show that, if \(h>1\) and \(L\) is general, the line bundles \(\sigma^*L\) are pairwise nonisomorphic as $\sigma$ varies in $\Sigma$. 

For any fixed \(\sigma \in \Sigma'\), let \(H \subset G\) be the subgroup fixing \(\sigma\).  The subvariety \(\sigma^*\Pic^dX \subset \Pic Z\) has dimension \(g\), contains \(\sigma^*L\), and is invariant under \(H\).  On the other hand, if \(\sigma^*\Pic^d X\) is invariant under a subgroup \(H'\) with \(H \subsetneq H' \subset G\), then it would be pulled back from \(\Pic (Z/H')\).  By the Riemann--Hurwitz formula, the genus of \(Z/H'\) is strictly smaller than \(g\).  Hence by dimension reasons, this containment is impossible and the \(\sigma^*L\) are pairwaise distinct.
 This shows the stability of $\alpha_* L$.

\section{Proof of Theorem~\ref{thm:gen}}\label{sec:our_degen}

Fix an \'etale cover \(\alpha^\et \colon Y' \to Y\) of degree \(r^\et \mid r\),
and a genus \(g > r(h - 1) + 1\).
We first explain how to construct a nodal cover \(\alpha_1 \colon X_1 \to Y\), whose primitive-\'etale factorization is 
\[X_1 \to Y' \xrightarrow{\alpha^\et} Y.\]
%In the subsequent subsections we will prove that \(X \to Y'\) is primitive
%(so that \(\alpha\) has the desired primitive-\'etale factorization),
%respectively the proof that \(\alpha_* V\) is stable for a general vector bundle \(V\) of any degree on \(X\).

The first step of our construction is to fix a cyclic \'etale cover \(\beta_0 \colon X_0 \to Y'\) of degree \(r^\pr = r / r^\et\).
Such covers correspond to points of order \(r^\pr\) in \(\text{Jac}(Y')\), which always exist.
Write \(\tau\colon X_0 \to X_0\) for the automorphism corresponding to the generator of \(\zz / r^\pr \zz\).
Let \(n \colonequals g - r(h-1) - 1\).  We then pick general points \(p_1, p_2, \ldots, p_{n} \in X_0\),
and let \(p_i' = \tau(p_i)\). 
Let \(X_1\) be the curve obtained from \(X_0\) by gluing every \(p_i\) to \(p_i'\) for \(1 \leq i \leq n\) and denote the normalization \( \nu \colon X_0 \to X_1\).  Let \(\beta_1 \colon X_1 \to Y'\) be the induced morphism, and let \(\alpha_1 \colonequals \alpha^\et \circ \beta_1\).
%We take our cover \(X \to Y'\) to be a smoothing of \(X_1 \to Y'\);
%composing with \(Y' \to Y\) defines the map \(\alpha \colon X \to Y\).

\begin{prop}\label{prop-primitive}
The cover \(\beta_1 \colon X_1 \to Y'\) is primitive.
\end{prop}
\begin{proof}
We must show that the pushforward \(\pi_1(X_1) \to \pi_1(Y')\) is surjective.
By \cite[Proposition 5.5.4(2)]{s09},
this is equivalent to the assertion that for all finite \'etale connected covers \(Y'' \to Y'\),
the fiber product \(X_1 \times_{Y'} Y''\) is connected.

Consider the dominant map \(\epsilon \colon X_0 \times_{Y'} Y'' \to X_1 \times_{Y'} Y''\).
By construction, \(\zz/r^\pr\zz\) acts transitively on the components of \(X_0 \times_{Y'} Y''\).
Therefore it suffices to show that for any component \(Z \subset X_0 \times_{Y'} Y''\),
the components \(\epsilon(Z)\) and \(\epsilon(\tau(Z))\) intersect.
Since \(Z \to X_0\) is surjective, \(Z\) contains a point of the form \((p_1, y'')\) for some \(y'' \in Y''\),
and so \(\tau(Z)\) contains \((p_1', y'')\).
Since \(\epsilon((p_1, y'')) = \epsilon((p_1', y''))\), the components \(\epsilon(Z)\) and \(\epsilon(\tau(Z))\) intersect as desired.
\end{proof}

By the theory of formal patching (see \cite[Lemma 5.6]{li}), the map \(\beta_1\) can be smoothed to a map \(\beta\colon X \to Y'\)  with a smooth domain \(X\).  Since being primitive is a deformation invariant, the resulting smoothing \(\beta\) is also primitive by Proposition~\ref{prop-primitive}.

\begin{prop} \label{pushforward-is-limit}
Given a vector bundle \(V\) on \(X_0\), the pushforward \(\nu_*V\) to \(X_1\) is a limit of vector bundles
of the same rank and slope \(\mu(V) + n\) on the smoothing \(X\).
\end{prop}
\begin{proof}
Let \(\mathcal{X} \to \Delta\) denote a family of smooth curves specializing
to \(X_1\) with smooth total space.
Consider the blowup at the nodes of \(X_1\). In this family, 
 the central fiber is the union of $X_0$ together with $n$ copies of $\pp^1$, where each $\pp^1$ is attached at the two preimages of the corresponding node under the normalization map $\nu$. These $\pp^1$s appear with multiplicity 2 in the central fiber. Make a base change of order 2 and normalize the total space to obtain a family $ \mathcal{X}^+ \to \Delta'$.  This is  a semistable family of smooth curves specializing to the union of $X_0$ with $n$ copies of $\pp^1$, where now the central fiber $X^+$ is reduced. Write \(\kappa \colon \mathcal{X}^+ \to \mathcal{X}' \colonequals \mathcal{X} \times_{\Delta}\Delta'\). 
The following diagram illustrates the maps that exist on the central fiber.

\begin{center}
\begin{tikzcd}
& X_0 \arrow[dl, hook] \arrow[dr, "\beta_0", swap] \arrow[drr, "\alpha_0", bend left = 15]\arrow[dd, "\nu", swap] \\
X^+ \arrow[dr, "\kappa|_{X^+}", swap] && Y' \arrow{r}{\alpha^{\et}} & Y \\
& X_1 \arrow{ur}{\beta_1} \arrow[urr, "\alpha_1", swap, bend right = 15]
\end{tikzcd}
\end{center}

Let \(V^+\) denote the vector bundle on \(X^+\) obtained by gluing
the vector bundle \(V\) on \(X_0\) to \(\O_{\pp^1}(1)^{\oplus \rk V}\) on each \(\pp^1\) component via any choice of gluing. (In fact the reader may check that any two choices result in isomorphic bundles.)  Since $V$ is locally free, \(\mathscr{E}{\kern -1.5pt}xt^i(V^+, V^+) = 0\) for all \(i > 0\). Thus, by the local-to-global $\Ext$ spectral sequence, we have that \(\Ext_{X^+}^2(V^+, V^+) =0\). By \cite[Theorem 7.3 (b)]{ha10}, the obstructions to extending the bundle $V^+$ to the whole family lie in $\Ext_{X^+}^2(V^+, V^+)$. Consequently, $V^+$ extends to a vector bundle \(\mathcal{V}^+\) on \(\mathcal{X}^+\). Observe that $\mathcal{V}^+$ has  rank $\rk V$, and by the constancy of the Euler characteristic in flat families, the  slope of  the restriction of $\mathcal{V}^+$ to the fibers is \(\mu(V^+)=  \mu(V) + n\).

 Now we claim that $\kappa_* \mathcal{V}^+|_{X_1} \simeq \nu_* V$. Once we establish this claim, we obtain that $\nu_* V$ is the limit of vector bundles of the same rank and slope $\mu(V) + n$ on the smooth fibers. 
 
Let \((\rk, \chi)(F)\) denote the rank and Euler characteristic of a sheaf \(F\), and  write \(\mathcal{X}'_t\) and \(\mathcal{X}^+_t\) for general fibers of their respective families.  We first show that \((\rk, \chi)(\kappa_* \mathcal{V}^+|_{X_1}) = (\rk, \chi)(\nu_*V)\).
By the constancy of the rank and the Euler characteristic in flat families, and the fact that \(\kappa\) is an isomorphism away from the central fiber, we have

\[(\rk, \chi)(\kappa_* \mathcal{V}^+|_{X_1}) = (\rk, \chi)(\kappa_* \mathcal{V}^+|_{\mathcal{X}'_t}) = (\rk, \chi)(\mathcal{V}^+|_{\mathcal{X}^+_t}) = (\rk, \chi)(V^+). \]
Furthermore, by considering the exact sequence for restriction to \(X_0\)
\[0 \to \bigoplus_{i=1}^n \O_{\pp^1}(-1)^{\oplus \rk(V)} \to V^+ \to V^+|_{X_0} = V \to 0,\]
we see that 
\((\rk, \chi)(V^+) = (\rk, \chi)(V)\).  Finally, \((\rk, \chi)(V) = (\rk, \chi)(\nu_*V)\), which proves that \((\rk, \chi)(\kappa_* \mathcal{V}^+|_{X_1}) = (\rk, \chi)(\nu_*V)\).

Hence, it suffices to construct a surjective map between  $\kappa_* \mathcal{V}^+|_{X_1}$ and $\nu_* V$.
Consider the exact sequence on \(\mathcal{X}^+\)
\[0 \to \mathcal{V}^+(-X_0) \to \mathcal{V}^+ \to \mathcal{V}^+|_{X_0} \to 0.\]
Pushing forward under \(\kappa\), we obtain
\[0 \to \kappa_*\mathcal{V}^+(-X_0) \to \kappa_*\mathcal{V}^+ \to \kappa_*(\mathcal{V}^+|_{X_0}) \to R^1\kappa_*\mathcal{V}^+(-X_0) \to \cdots\]
Observe that  \(\kappa_*(\mathcal{V}^+|_{X_0}) \simeq \nu_*V\), and that the map \(\kappa_*\mathcal{V}^+ \to \kappa_*(\mathcal{V}^+|_{X_0}) \simeq \nu_*V\) factors through \((\kappa_*\mathcal{V}^+)|_{X_1}\).  Hence it suffices to show that \(R^1\kappa_*\mathcal{V}^+(-X_0) = 0\). 

Since \(\kappa\) is an isomorphism away from the nodes of the \(X_1\), the sheaf \(R^1\kappa_*\mathcal{V}^+(-X_0)\) is supported on the nodes of \(X_1\).  It therefore suffices to show that its completion at every node \(p\) of \(X_1\) vanishes.  For this, we use the theorem on formal functions, which states that
\[R^1\kappa_*\mathcal{V}^+(-X_0)^\wedge_p \simeq \varprojlim_{n} H^1(\mathcal{V}^+(-X_0)|_{n \cdot\pp^1}),\]
where \(\pp^1 = \kappa^{-1}(p)\) is a Cartier divisor on \(\mathcal{X}^+\).
It thus suffices to show that \(H^1(\mathcal{V}^+(-X_0)|_{n \cdot\pp^1})=0\) for all \(n\).  For this, we use induction on \(n\), with base case \(n=0\), which is clear.  For the inductive step, we use the exact sequence for restriction to \(n \cdot \pp^1\)
\[0 \to \mathcal{V}^+(-X_0 - n \cdot\pp^1)|_{\pp^1} \to \mathcal{V}^+(-X_0)|_{(n+1) \cdot \pp^1} \to \mathcal{V}^+(-X_0)|_{n \cdot \pp^1}\to 0.\]
Since \(\mathcal{V}^+(-X_0 - n \cdot\pp^1)|_{\pp^1} \simeq \O_{\pp^1}(2n-1)^{\oplus \rk V}\), which has vanishing \(h^1\), we conclude by induction that the middle term has vanishing \(h^1\).
%  To see the claimed isomorphism, we use the  natural base change map $$\psi \colon \kappa_* \mathcal{V}^+|_{X_1} \to (\kappa|_{\kappa^{-1}(X_1)})_* V^+.$$ First, we claim that $(\kappa|_{\kappa^{-1}(X_1}))_* V^+ \simeq \nu_* V$.  To see this claim, use the exact sequence
%  $$0 \to \bigoplus_{i=1}^n \O_{\pp^1}(-1)^{\oplus \rk(V)} \to V^+ \to V^+|_{X_0} = V \to 0.$$ Apply $\kappa_*$ to this sequence obtain
%  $$0 \to \kappa_*\bigoplus_{i=1}^n \O_{\pp^1}(-1)^{\oplus \rk(V)} \to \kappa_*  V^+ \to \kappa_* V = \nu_* V \to R^1 \kappa_*\bigoplus_{i=1}^n \O_{\pp^1}(-1)^{\oplus \rk(V)} \to \cdots$$
%  Since the outer two terms vanish, we get $(\kappa|_{\kappa^{-1}(X_1}))_* V^+ \simeq \nu_* V$. 
% Next, we show that the map $\psi$ is an isomorphism. Away from the nodes of $X_1$, the map $\kappa$ is an isomorphism and therefore so is $\psi$. To check that it is also an isomorphism at the nodes, we use the formal function theorem. Let $p$ denote a node of $X_1$ and let $C$ denote the $\pp^1$ over $p$.
%  We have that $$\varprojlim_n \kappa_* \mathcal{V}|_{np} \simeq \varprojlim_n H^0(\mathcal{V}|_{nC})$$
\end{proof}

\begin{proof}[{Proof of Theorem \ref{thm:gen}}]
Let $\sH$ be an irreducible component of the Hurwitz space $\sH_{r, g}(Y)$.  Assume that the corresponding  covers have primitive-\'etale factorization
\[X \to Y' \xrightarrow{\alpha^\et} Y.\]
Let \(\beta_0 \colon X_0 \to Y'\) be the cyclic \'etale cover constructed above and let
$$\alpha_1 \colon X_1 \xrightarrow{\beta_1} Y' \xrightarrow{\alpha^\et} Y$$ be the cover constructed above by gluing points in the fibers of \(\beta_0 \colon X_0 \to Y'\).

Let
\(V\) be a general vector bundle on \(X_0\) of arbitrary degree and rank. By \cite[Proposition 4.1]{b00} (see \S \ref{sec-etale}), the pushforward ${\alpha_0}_*V$ is semistable if $h=1$ and stable if $h \geq 2$. Since $$\alpha_0 \ = \ \alpha^{\et} \circ \beta_0 \ = \ \alpha^{\et} \circ \beta_1 \circ \nu \ = \ \alpha_1 \circ \nu,$$ we conclude that ${\alpha_1}_* \nu_* V$ is semistable if $h=1$ and stable if $h \geq 2$. 

By Proposition~\ref{pushforward-is-limit},
the pushforward \(\nu_*V\) on \(X_1\) is a limit of vector bundles
of arbitrary degree and rank on a smoothing \(X \to Y' \to Y\). The theorem follows by the openness of (semi)stability.
\end{proof}

\bibliographystyle{plain}

\end{document}